\documentclass[11pt]{article}
\usepackage[T1]{fontenc}
\usepackage[utf8]{inputenc}
\usepackage{lmodern}
\usepackage{amsmath,amssymb,amsthm}
\usepackage[margin=30mm]{geometry}
\usepackage[hidelinks]{hyperref}
\usepackage{authblk}

\newtheorem{theorem}{Theorem}[section]
\newtheorem{lemma}[theorem]{Lemma}
\newtheorem{proposition}[theorem]{Proposition}
\newtheorem{corollary}[theorem]{Corollary}
\numberwithin{equation}{section}

\title{\textbf{A smooth non-compactly generated category}}
\author[1]{A.S. Frolov}
\affil[1]{Higher School of Modern Mathematics, MIPT}
\affil[]{\textit{frolov.aleksandr@phystech.edu}}

\date{\today}

\begin{document}
\maketitle
\begin{abstract}
Let $A=\mathbb{Q}[t^{\frac{1}{2^n}}:n\geq 0]$ and consider
$A[u,u^{-1};\sigma]$, a skew Laurent polynomial ring, where $\sigma(t^q)=t^{2q}$.
We study the kernel of derived extension of scalars along
\[
 A[u,u^{-1};\sigma]\longrightarrow\mathbb{Q}[u,u^{-1}],
\]
the quotient by $\mathfrak{m}A[u,u^{-1};\sigma]$, with
$\mathfrak{m}=(t^{\frac{1}{2^n}}:n\geq 0)$.
A finite free resolution of
$\mathfrak{m}A[u,u^{-1};\sigma]$ over
\[
 A[u,u^{-1};\sigma]\otimes_{\mathbb{Q}}
 \bigl(A[u,u^{-1};\sigma]\bigr)^{\mathrm{op}}
\]
proves smoothness. A specialization inequality for matrix ranks supplies
a nonzero exact colimit-preserving functor that vanishes on every compact
object, proving failure of compact generation.
\end{abstract}

\section{Introduction}

Let
\[
 A=\mathbb{Q}[t^{\frac{1}{2^n}}:n\geq 0],\qquad
 \mathfrak{m}=(t^{\frac{1}{2^n}}:n\geq 0),
 \qquad \sigma(t^q)=t^{2q}.
\]
We consider the category $\mathcal{C}$
\[
 \mathcal{C}:= \ker\left(D\bigl(A[u,u^{-1};\sigma]\bigr)
 \xrightarrow{-\otimes_{A[u,u^{-1};\sigma]}^{\mathbf{L}}
                    \mathbb{Q}[u,u^{-1}]}
 D\bigl(\mathbb{Q}[u,u^{-1}]\bigr)\right).
\]
Here $A[u,u^{-1};\sigma]$ is a skew Laurent polynomial ring. Throughout the paper, ``$\mathcal{C}$'' refers to this kernel.
The quotient used here is
\[
 \frac{A[u,u^{-1};\sigma]}{\mathfrak{m}A[u,u^{-1};\sigma]}
 \cong\mathbb{Q}[u,u^{-1}].
\]
The algebra $A[u,u^{-1};\sigma]$ consists of finite sums $\sum_i a_i u^i$,
with multiplication
\[
 (au^i)(bu^j)=a\sigma^i(b)u^{i+j}.
\]

\begin{theorem}\label{thm:main}
The category $\mathcal{C}$ is dualizable and smooth, and it is not compactly generated.
\end{theorem}

This example is a counterexample to the conjecture \cite[Conjecture 0.6]{ramzi}.
\vspace{5pt}

The proof has two main ingredients: a finite bimodule resolution of
$\mathfrak{m}A[u,u^{-1};\sigma]$ and a nonzero exact colimit-preserving
functor annihilating every compact object of the category $\mathcal{C}$.
Sections~2 and~4 establish dualizability and smoothness, while Section~5
proves failure of compact generation.
\vspace{5pt}

In the subsequent work \cite{geom}, we prove the following theorems \textbf{proposed by A.I. Efimov}:

\begin{theorem}
    Let $X$ be a separated scheme of finite type over finite field $\mathbb{F}_p$ and let $Y\subset X$ be a closed subscheme. Denote by $\varphi$ the endofunctor of $\mathrm{QCoh}(X)$ given by pullback along the Frobenius morphism, and use the same notation for the corresponding endofunctor of $\mathrm{QCoh}(Y)$. Suppose that $X\backslash Y\ne \varnothing$ is smooth. Then the category $$\ker\left(\mathrm{QCoh}(X)_{h\varphi}\to \mathrm{QCoh}(Y)_{h\varphi}\right)$$ is a smooth, dualizable stable category that is not compactly generated. 
\end{theorem}

\begin{theorem}
    There is a smooth proper stable category which provides a counterexample to the naive telescope conjecture. 
\end{theorem}

\textbf{Statement on AI use.} ChatGPT 6 Astra was used to assist with the proof of the main
results and the preparation of an initial draft of the manuscript. The arguments and text have been reviewed, checked, and
substantially revised by the authors, who are responsible for the final manuscript.

\section*{Acknowledgement}

The author is grateful to Sasha Efimov for a useful discussion and a suggestion for further developing the work, which will appear in \cite{geom}.

\section{The ideal and the derived kernel}

\begin{lemma}
The ideal $\mathfrak{m}A[u,u^{-1};\sigma]$ is idempotent and flat on both
sides over $A[u,u^{-1};\sigma]$. Moreover,
\[
\begin{aligned}
 \mathbb{Q}[u,u^{-1}]\otimes_{A[u,u^{-1};\sigma]}^{\mathbf{L}}
       \mathbb{Q}[u,u^{-1}]&\simeq\mathbb{Q}[u,u^{-1}],\\
 \mathfrak{m}A[u,u^{-1};\sigma]\otimes_{A[u,u^{-1};\sigma]}^{\mathbf{L}}
       \mathbb{Q}[u,u^{-1}]&\simeq 0,\\
 \mathfrak{m}A[u,u^{-1};\sigma]\otimes_{A[u,u^{-1};\sigma]}^{\mathbf{L}}
       \mathfrak{m}A[u,u^{-1};\sigma]&\simeq
       \mathfrak{m}A[u,u^{-1};\sigma].
\end{aligned}
\]
\end{lemma}

\begin{proof}
The ideal $\mathfrak{m}$ is flat because
\[
 \mathfrak{m}=\bigcup_{n\geq 0}t^{\frac{1}{2^n}}A
\]
is an increasing union of free rank-one $A$-modules. Moreover,
\[
 t^{\frac{1}{2^n}}=\left(t^{\frac{1}{2^{n+1}}}\right)^2
\]
implies $\mathfrak{m}^2=\mathfrak{m}$.

Since $\sigma^{\pm1}$ preserve $\mathfrak{m}$,
\[
\begin{gathered}
 \mathfrak{m}A[u,u^{-1};\sigma]=A[u,u^{-1};\sigma]\mathfrak{m},\\
 \bigl(\mathfrak{m}A[u,u^{-1};\sigma]\bigr)^2
       =\mathfrak{m}A[u,u^{-1};\sigma].
\end{gathered}
\]
The identifications
\[
\begin{aligned}
 \mathfrak{m}A[u,u^{-1};\sigma]
       &\cong\mathfrak{m}\otimes_A A[u,u^{-1};\sigma],\\
 \mathfrak{m}A[u,u^{-1};\sigma]
       &\cong A[u,u^{-1};\sigma]\otimes_A\mathfrak{m}
\end{aligned}
\]
show that $\mathfrak{m}A[u,u^{-1};\sigma]$ is flat as both a right and a
left $A[u,u^{-1};\sigma]$-module.

Consequently,
\begin{equation}\label{eq:flat}
\begin{aligned}
 0\longrightarrow\mathfrak{m}A[u,u^{-1};\sigma]
  &\longrightarrow A[u,u^{-1};\sigma]\\
  &\longrightarrow\mathbb{Q}[u,u^{-1}]\longrightarrow 0
\end{aligned}
\end{equation}
is a flat resolution of $\mathbb{Q}[u,u^{-1}]$ on either side. Since
\[
\begin{aligned}
 \mathfrak{m}A[u,u^{-1};\sigma]\otimes_{A[u,u^{-1};\sigma]}
       \mathbb{Q}[u,u^{-1}]
   & =\frac{\mathfrak{m}A[u,u^{-1};\sigma]}
                {\bigl(\mathfrak{m}A[u,u^{-1};\sigma]\bigr)^2}\\
   & =0,
\end{aligned}
\]
we obtain
\begin{equation}\label{eq:tensor-identities}
\begin{aligned}
 \mathbb{Q}[u,u^{-1}]\otimes_{A[u,u^{-1};\sigma]}^{\mathbf{L}}
       \mathbb{Q}[u,u^{-1}]&\simeq\mathbb{Q}[u,u^{-1}],\\
 \mathfrak{m}A[u,u^{-1};\sigma]\otimes_{A[u,u^{-1};\sigma]}^{\mathbf{L}}
       \mathbb{Q}[u,u^{-1}]&\simeq 0,\\
 \mathfrak{m}A[u,u^{-1};\sigma]\otimes_{A[u,u^{-1};\sigma]}^{\mathbf{L}}
       \mathfrak{m}A[u,u^{-1};\sigma]&\simeq
       \mathfrak{m}A[u,u^{-1};\sigma].
\end{aligned}
\end{equation}
\end{proof}

\section{A finite bimodule resolution}

We regard bimodules as left modules over
\[
 A[u,u^{-1};\sigma]\otimes_{\mathbb{Q}}\bigl(A[u,u^{-1};\sigma]\bigr)^{\mathrm{op}}.
\]
The relevant tensor elements satisfy
\[
\begin{aligned}
 &(u\otimes(u^{-1})^{\mathrm{op}})(t\otimes1)
   =(t^2\otimes1)(u\otimes(u^{-1})^{\mathrm{op}}),\\
 &(u\otimes(u^{-1})^{\mathrm{op}})(1\otimes t^{\mathrm{op}})
   =(1\otimes(t^2)^{\mathrm{op}})(u\otimes(u^{-1})^{\mathrm{op}}).
\end{aligned}
\]

\begin{proposition}
The left
$A[u,u^{-1};\sigma]\otimes_{\mathbb{Q}}\bigl(A[u,u^{-1};\sigma]\bigr)^{\mathrm{op}}$-module
$\mathfrak{m}A[u,u^{-1};\sigma]$ is perfect. More precisely, there is an exact
sequence
\begin{equation}\label{eq:resolution}
\begin{aligned}
0&\longrightarrow
 A[u,u^{-1};\sigma]\otimes_{\mathbb{Q}}\bigl(A[u,u^{-1};\sigma]\bigr)^{\mathrm{op}}\\
 &\xrightarrow{\alpha}
 \bigl(A[u,u^{-1};\sigma]\otimes_{\mathbb{Q}}\bigl(A[u,u^{-1};\sigma]\bigr)^{\mathrm{op}}\bigr)^2\\
 &\xrightarrow{\beta}
 A[u,u^{-1};\sigma]\otimes_{\mathbb{Q}}\bigl(A[u,u^{-1};\sigma]\bigr)^{\mathrm{op}}\\
 &\xrightarrow{\varepsilon}\mathfrak{m}A[u,u^{-1};\sigma]
 \longrightarrow0,
\end{aligned}
\end{equation}
where
\begin{equation}\label{eq:maps}
\begin{aligned}
\alpha(h)&=
 \begin{pmatrix}
 h\bigl[u\otimes(u^{-1})^{\mathrm{op}}-t^2\otimes1-t\otimes t^{\mathrm{op}}\bigr]\\
 -h\bigl(t^2\otimes1-1\otimes(t^2)^{\mathrm{op}}\bigr)
 \end{pmatrix},\\[3pt]
\beta(p,q)&=p(t\otimes1-1\otimes t^{\mathrm{op}})\\
 &\qquad+q(u\otimes(u^{-1})^{\mathrm{op}}-t\otimes1),\\
\varepsilon(h)&=h\cdot t.
\end{aligned}
\end{equation}
The coefficients multiply on the right, making these maps left-linear over
$A[u,u^{-1};\sigma]\otimes_{\mathbb{Q}}\bigl(A[u,u^{-1};\sigma]\bigr)^{\mathrm{op}}$.
\end{proposition}

\begin{proof}
For exactness, consider the skew polynomial ring
$\mathbb{Q}[t,s][v;\tau]$, where
\[
 \tau(t)=t^2,\qquad \tau(s)=s^2,
\]
and its left module $t\mathbb{Q}[t]$.
Both $t,s$ act by multiplication by $t$, while
\[
v\cdot f(t)=f(t^2).
\]
The generator $t\in t\mathbb{Q}[t]$ satisfies
\[
(t-s)t=0,\qquad (v-t)t=0.
\]
Consider
\begin{equation}\label{eq:h-resolution}
\begin{aligned}
0&\longrightarrow\mathbb{Q}[t,s][v;\tau]
 \xrightarrow{a}\bigl(\mathbb{Q}[t,s][v;\tau]\bigr)^2\\
 &\xrightarrow{b}\mathbb{Q}[t,s][v;\tau]
 \xrightarrow{\epsilon}t\mathbb{Q}[t]\longrightarrow0,
\end{aligned}
\end{equation}
with
\[
\begin{aligned}
a(h)&=\bigl(h[v-t(t+s)],-h(t^2-s^2)\bigr),\\
b(p,q)&=p(t-s)+q(v-t),\\
\epsilon(h)&=h\cdot t.
\end{aligned}
\]
The skew relation gives
\[
[v-t(t+s)](t-s)-(t^2-s^2)(v-t)=0,
\]
so $ba=0$.

Every element of $\mathbb{Q}[t,s][v;\tau]$ has a unique normal form
\[
\sum_{j\geq0}f_j(t,s)v^j.
\]
Because $\tau$ is injective, $\mathbb{Q}[t,s][v;\tau]$ is a domain.
Right multiplication by $v-t(t+s)$ increases $v$-degree by one,
so $a$ is injective.

For a nonzero relation
\[
p(t-s)+q(v-t)=0,
\]
degree comparison gives
\[
\deg_v p=n\geq1,\qquad \deg_v q=n-1.
\]
If $p_n,q_{n-1}$ are the leading coefficients, then
\[
q_{n-1}=-p_n(t^{2^n}-s^{2^n}).
\]
Subtracting $a(p_nv^{n-1})$ cancels both leading terms. Induction proves
\[
\ker b=\operatorname{im}a.
\]
Finally, modulo the left ideal $\mathbb{Q}[t,s][v;\tau](v-t)$,
\[
v^j\equiv v^{j-1}t=t^{2^{j-1}}v^{j-1}.
\]
Repeated reduction removes every $v$. Reduction modulo
$\mathbb{Q}[t,s][v;\tau](t-s)$ then leaves a polynomial $f(t)$,
whose image in $t\mathbb{Q}[t]$ is $f(t)t$. Since $\mathbb{Q}[t]$
is a domain, this vanishes exactly when $f=0$. Hence
\[
\ker\epsilon=\operatorname{im}b,\qquad
\text{which proves \eqref{eq:h-resolution}.}
\]

Next consider the skew Laurent ring
\[
 (A\otimes_{\mathbb{Q}}A)[v,v^{-1};\sigma\otimes\sigma].
\]
As right $\mathbb{Q}[t,s][v;\tau]$-modules,
\begin{equation}\label{eq:u-union}
(A\otimes_{\mathbb{Q}}A)[v,v^{-1};\sigma\otimes\sigma]
=\bigcup_{n\geq0}v^{-n}\mathbb{Q}[t,s][v;\tau].
\end{equation}
Indeed, left multiplication by a sufficiently large power of $v$ clears
all negative $v$-powers and makes every dyadic exponent integral. Each
$v^{-n}\mathbb{Q}[t,s][v;\tau]$ is right-free of rank one, so
$(A\otimes_{\mathbb{Q}}A)[v,v^{-1};\sigma\otimes\sigma]$
is right-flat over $\mathbb{Q}[t,s][v;\tau]$.

Moreover,
\[
(A\otimes_{\mathbb{Q}}A)[v,v^{-1};\sigma\otimes\sigma]
\otimes_{\mathbb{Q}[t,s][v;\tau]}t\mathbb{Q}[t]
\cong\mathfrak{m}.
\]
The underlying direct system has transition maps
\[
f(t)\longmapsto f(t^2),
\]
and its $n$-th stage embeds into $A$ through
\[
f(t)\longmapsto f(t^{\frac{1}{2^n}}).
\]
Embed $(A\otimes_{\mathbb{Q}}A)[v,v^{-1};\sigma\otimes\sigma]$ in
$A[u,u^{-1};\sigma]\otimes_{\mathbb{Q}}\bigl(A[u,u^{-1};\sigma]\bigr)^{\mathrm{op}}$ by
\[
t\longmapsto t\otimes1,\qquad
s\longmapsto1\otimes t^{\mathrm{op}},\qquad
v\longmapsto u\otimes(u^{-1})^{\mathrm{op}}.
\]
There is a decomposition as right modules over
$(A\otimes_{\mathbb{Q}}A)[v,v^{-1};\sigma\otimes\sigma]$:
\begin{equation}\label{eq:e-decomp}
\begin{aligned}
&A[u,u^{-1};\sigma]\otimes_{\mathbb{Q}}\bigl(A[u,u^{-1};\sigma]\bigr)^{\mathrm{op}}\\
&\qquad=\bigoplus_{j\in\mathbb{Z}}(u^j\otimes1)
 (A\otimes_{\mathbb{Q}}A)[v,v^{-1};\sigma\otimes\sigma].
\end{aligned}
\end{equation}
Multiplication identifies
\[
\begin{aligned}
&\bigl(A[u,u^{-1};\sigma]\otimes_{\mathbb{Q}}
 \bigl(A[u,u^{-1};\sigma]\bigr)^{\mathrm{op}}\bigr)
 \otimes_{(A\otimes_{\mathbb{Q}}A)[v,v^{-1};\sigma\otimes\sigma]}
 \mathfrak{m}\\
&\qquad\cong\mathfrak{m}A[u,u^{-1};\sigma].
\end{aligned}
\]
Thus both induction steps are exact. Applying them to
\eqref{eq:h-resolution} gives \eqref{eq:resolution}, proving that
$\mathfrak{m}A[u,u^{-1};\sigma]$ is perfect as a left module over
$A[u,u^{-1};\sigma]\otimes_{\mathbb{Q}}\bigl(A[u,u^{-1};\sigma]\bigr)^{\mathrm{op}}$.
\end{proof}

\section{Smoothness}

Denote by $i\colon \mathcal{C}\hookrightarrow D(A[u,u^{-1};\sigma])$ the inclusion and by $r$ its right adjoint. It is easy to see $$r\circ i\simeq \operatorname{id}_{\mathcal{C}}, \ i\circ r\simeq - \otimes_{A[u,u^-1; \sigma]}^\mathbb{L} \mathfrak{m} A[u,u^{-1};\sigma],$$ so $\mathcal{C}$ is dualizable and $r$ preserves filtered colimits.

\begin{theorem}\label{thm:smoothness}
The category $\mathcal{C}$ is smooth over $\mathbb{Q}$.
\end{theorem}

\begin{proof}

Define
\[
\begin{aligned}
J:\operatorname{Fun}^L_\mathbb{Q}(\mathcal{C},\mathcal{C})
&\longrightarrow\operatorname{Fun}^L_\mathbb{Q}(D(A[u,u^{-1};\sigma]),D(A[u,u^{-1};\sigma])),\\
F&\longmapsto iFr.
\end{aligned}
\]

The functor $J$ preserves colimits and is fully faithful. Indeed,
postcomposition with $i$ is fully faithful, and
\[
\operatorname{Lan}_iF\simeq Fr
\]
because $i\dashv r$. Restricting along $i$ recovers $Fri\simeq F$.

Derived Eilenberg--Watts identifies continuous endofunctors of
$D(A[u,u^{-1};\sigma])$ with derived $A[u,u^{-1};\sigma]$-bimodules:
\[
B\longleftrightarrow
\bigl(X\mapsto X\otimes_{A[u,u^{-1};\sigma]}^{\mathbf{L}}B\bigr).
\]
Under this identification, the image under $J$ of the identity functor
of the category $\mathcal{C}$ satisfies
\[
 ir\longleftrightarrow\mathfrak{m}A[u,u^{-1};\sigma].
\]
By \eqref{eq:resolution}, this object is compact.

\end{proof}

\section{Failure of compact generation}

We use the fraction field $K=\operatorname{Frac}(A)$, the skew Laurent ring
$K[u,u^{-1};\sigma]$, and the field
\[
    \mathbb{Q}(u)=\operatorname{Frac}\bigl(\mathbb{Q}[u,u^{-1}]\bigr).
\]
For an Ore domain, $\operatorname{Frac}$ denotes its classical division
ring of fractions.

\begin{lemma}
The ring $K[u,u^{-1};\sigma]$ is an Ore domain on both
sides. Its division ring of fractions
$\operatorname{Frac}(K[u,u^{-1};\sigma])$ is flat as both
a left and a right $A[u,u^{-1};\sigma]$-module.
\end{lemma}

\begin{proof}
The Ore conditions for $K[u,u^{-1};\sigma]$ can be
checked by dimension counting. Given skew polynomials $f,g$ of degrees
$m,n$, their right multiples of degree at most $d$, for $d\geq m+n$, form
subspaces of dimensions
\[
    d-m+1,\qquad d-n+1
\]
inside a right $K$-vector space of dimension $d+1$.
They therefore have a nonzero intersection. This gives common right
multiples. Left vector spaces give common left multiples.

The division ring
$\operatorname{Frac}(K[u,u^{-1};\sigma])$ is flat on
both sides over $A[u,u^{-1};\sigma]$. For example,
\[
    K[u,u^{-1};\sigma]
    =\bigcup_{0\neq f\in A}f^{-1}A[u,u^{-1};\sigma]
\]
is a directed union of free rank-one right $A[u,u^{-1};\sigma]$-modules.
The corresponding left-module statement uses $A[u,u^{-1};\sigma]f^{-1}$.
The same Ore-fraction argument shows that
\[
 \operatorname{Frac}(K[u,u^{-1};\sigma])
\]
is flat on both sides over $K[u,u^{-1};\sigma]$.
\end{proof}

\begin{lemma}
For every finite matrix $Q$ over $A[u,u^{-1};\sigma]$, we have
\begin{equation}\label{eq:rank}
    \operatorname{rank}_{\operatorname{Frac}(K[u,u^{-1};\sigma])}(Q)
    \geq \operatorname{rank}_{\mathbb{Q}(u)}(\overline{Q}),
\end{equation}
where $\overline{Q}$ is obtained by setting every positive power of $t$
equal to zero.
\end{lemma}

\begin{proof}
Define the valuation
\[
    \nu\left(\sum_q a_qt^q\right)=\min\{q:a_q\neq 0\},
    \qquad \nu(f/g)=\nu(f)-\nu(g).
\]
Its valuation ring is $A_{\mathfrak{m}}$, with residue field $\mathbb{Q}$.
Since
\[
    \nu(\sigma(f))=2\nu(f),
\]
both $\sigma$ and $\sigma^{-1}$ preserve $A_{\mathfrak{m}}$ and its maximal
ideal. Thus reduction gives
\begin{equation}\label{eq:reduction}
    A_{\mathfrak{m}}[u,u^{-1};\sigma]\longrightarrow\mathbb{Q}[u,u^{-1}].
\end{equation}
Choose rows of $\overline{Q}$ independent over $\mathbb{Q}(u)$. If their
lifts were dependent over
\[
 \operatorname{Frac}(K[u,u^{-1};\sigma]),
\]
there would be a nonzero left relation between them.

Clear common left Ore denominators so its coefficients lie in
$K[u,u^{-1};\sigma]$. Let $q$ be the smallest valuation
of any nonzero Laurent coefficient occurring in the relation.
Multiplication of the entire relation on the left by $t^{-q}$ makes every
coefficient valuation nonnegative and at least one zero.

Reduction via \eqref{eq:reduction} then gives a nonzero relation between
the specialized rows, contradicting their independence. This proves
\eqref{eq:rank}.
\end{proof}

\begin{lemma}
For every finitely generated projective right $A[u,u^{-1};\sigma]$-module
$P$, the scalar extensions
\[
\begin{gathered}
    P\otimes_{A[u,u^{-1};\sigma]}
    \operatorname{Frac}(K[u,u^{-1};\sigma]),\\
    P\otimes_{A[u,u^{-1};\sigma]}\mathbb{Q}(u)
\end{gathered}
\]
have equal dimensions over their respective division rings. The rank of a
map between such modules over
\[
 \operatorname{Frac}(K[u,u^{-1};\sigma])
\]
is at least its rank over $\mathbb{Q}(u)$.
\end{lemma}

\begin{proof}
For a finitely generated projective $A[u,u^{-1};\sigma]$-module $P$, choose
an idempotent matrix $e$ representing it. Apply \eqref{eq:rank} to $e$ and
$1-e$. Their ranks sum to the same ambient free rank over the two division rings, so
both inequalities must be equalities:
\begin{equation}\label{eq:dimension}
\begin{aligned}
    &\dim_{\operatorname{Frac}(K[u,u^{-1};\sigma])}\\[-2pt]
    &\quad\bigl(P\otimes_{A[u,u^{-1};\sigma]}
      \operatorname{Frac}(K[u,u^{-1};\sigma])\bigr)\\
    &\qquad=\dim_{\mathbb{Q}(u)}
      \bigl(P\otimes_{A[u,u^{-1};\sigma]}\mathbb{Q}(u)\bigr).
\end{aligned}
\end{equation}
By choosing projective complements, the rank inequality also applies to
maps between finitely generated projectives.
\end{proof}

\begin{corollary}
For any perfect complex $P$,
\begin{equation}\label{eq:vanishing}
\begin{gathered}
    P\otimes_{A[u,u^{-1};\sigma]}^{\mathbf{L}}\mathbb{Q}[u,u^{-1}]\simeq 0\\
    \Longrightarrow\quad
    P\otimes_{A[u,u^{-1};\sigma]}^{\mathbf{L}}
    \operatorname{Frac}(K[u,u^{-1};\sigma])\simeq 0.
\end{gathered}
\end{equation}
\end{corollary}

\begin{proof}
Represent $P$ by a bounded complex of finitely generated projectives. Its
extension to $\mathbb{Q}(u)$ is exact. In a fixed degree, suppose the term
has dimension $n$, and the incoming and outgoing differentials have ranks
$a,b$. Exactness means
\[
    a+b=n.
\]
Over $\operatorname{Frac}(K[u,u^{-1};\sigma])$, the
term still has dimension $n$, by \eqref{eq:dimension}, while the two
differential ranks are at least $a,b$, by \eqref{eq:rank}. Since their
composite is zero, their ranks sum to at most $n$. Equality follows, so the
complex is also exact over
$\operatorname{Frac}(K[u,u^{-1};\sigma])$.
\end{proof}

\begin{theorem}\label{thm:not-compactly-generated}
The category $\mathcal{C}$ is not compactly generated.
\end{theorem}

\begin{proof}
Define a functor $G$ from the category $\mathcal{C}$ to
$D(\operatorname{Frac}(K[u,u^{-1};\sigma]))$ by
\[
    G(X)=iX\otimes_{A[u,u^{-1};\sigma]}^{\mathbf{L}}
    \operatorname{Frac}(K[u,u^{-1};\sigma]).
\]
If $X$ is compact in the category $\mathcal{C}$, then $iX$ is compact in
$D(A[u,u^{-1};\sigma])$, because the right adjoint $r$ preserves filtered
colimits: the adjunction identifies the mapping space from $iX$ to any
object $Y$ of $D(A[u,u^{-1};\sigma])$ with the mapping space from $X$ to
$r(Y)$ in the category $\mathcal{C}$.
Thus $iX$ is perfect. Since
$iX\otimes_{A[u,u^{-1};\sigma]}^{\mathbf{L}}\mathbb{Q}[u,u^{-1}]=0$,
implication \eqref{eq:vanishing} gives
\[
    G(X)=0.
\]
However,
\begin{equation}\label{eq:nonzero}
\begin{gathered}
    \mathfrak{m}A[u,u^{-1};\sigma]\text{ belongs to the category $\mathcal{C}$},\\
    G\bigl(\mathfrak{m}A[u,u^{-1};\sigma]\bigr)
    \simeq\operatorname{Frac}(K[u,u^{-1};\sigma])\neq 0.
\end{gathered}
\end{equation}
To check the second assertion, $t$ acts as zero on
$\mathbb{Q}[u,u^{-1}]$ and is invertible in
\[
 \operatorname{Frac}(K[u,u^{-1};\sigma]).
\]
Hence
\[
    \mathbb{Q}[u,u^{-1}]\otimes_{A[u,u^{-1};\sigma]}
    \operatorname{Frac}(K[u,u^{-1};\sigma])=0.
\]
Tensoring \eqref{eq:flat} with the flat left $A[u,u^{-1};\sigma]$-module
$\operatorname{Frac}(K[u,u^{-1};\sigma])$ gives
\[
\begin{gathered}
    \mathfrak{m}A[u,u^{-1};\sigma]\otimes_{A[u,u^{-1};\sigma]}
    \operatorname{Frac}(K[u,u^{-1};\sigma])\\
    \xrightarrow{\sim}
    \operatorname{Frac}(K[u,u^{-1};\sigma]),
\end{gathered}
\]
which also computes the derived tensor product.

The kernel of $G$ is a localizing subcategory containing every compact
object of the category $\mathcal{C}$, but it does not contain
$\mathfrak{m}A[u,u^{-1};\sigma]$. Therefore those compact objects cannot
generate the category $\mathcal{C}$.
\end{proof}

\begin{proof}[Proof of Theorem~\ref{thm:main}]
Smoothness is
Theorem~\ref{thm:smoothness}, and failure of compact generation is
Theorem~\ref{thm:not-compactly-generated}.
\end{proof}

\end{document}